\documentclass[11pt]{article}

\usepackage{geometry}
\usepackage{amsmath,amsthm,amssymb,latexsym,amstext,amsfonts}
\usepackage{pst-node}
\usepackage{caption}
\usepackage{makeidx}
\usepackage{pstricks-add}
\usepackage{mathdots}
\usepackage{braket}

\usepackage{graphicx,pgf}
\usepackage{pstricks,pst-grad,pst-plot,pst-node}
\usepackage{tikz}
\usetikzlibrary{snakes}
\usetikzlibrary{arrows}
\usepackage{epic}

\newtheorem{proposition}{Proposition}
\newtheorem{theorem}[proposition]{Theorem}

\newtheorem{lemma}[proposition]{Lemma}

\newcommand{\N}{\mathbb{N}}

\newcommand{\A}{\mathbb{A}}
\newcommand{\B}{\mathbb{B}}
\newcommand{\Sl}{\mathbb{S}}
\newcommand{\La}{\mathbb{L}}

\newcommand{\C}{\mathbb{C}}
\newcommand{\G}{\mathbb{G}}

\def \RR{{\mathcal R}}

\def \HH{{\mathcal H}}
\def \LL{{\mathcal L}}
\def \DD{{\mathcal D}}

\def \FF{{\mathcal F}}
\def \CC{{\mathcal C}}

\def \FF{{\mathcal F}}
\def \UU{{\mathcal U}}

\def \AA{{\mathcal A}}
\def \BB{{\mathcal B}}

\title{\emph{ON THE COSET STRUCTURE \\ OF A SKEW LATTICE}}
\author{\textit{Jo\~ao Pita Costa}\footnote{The author thanks the support
of Funda\c c\~ao para a Ci\^ encia e Tecnologia with the reference SFRH / BD / 36694
/ 2007.}\bigskip\\
University of Ljubljana,\\
Faculty of Mathematics and Physics,\\
Jadranska 19, 1000 Ljubljana, Slovenia.\\
 joaopitacosta@gmail.com}
\date{\today}

\begin{document}

\maketitle

\begin{abstract}

The class of skew lattices can be seen as an algebraic category. It models an algebraic theory in the category of Sets where the Green's relation $D$ is a congruence describing an adjunction to the category of Lattices. In this paper we will discuss the relevance of this approach, revisit some known decompositions and relate the order structure of a skew lattice with its coset structure that describes the internal coset decomposition of the respective skew lattice.

\bigskip

\medskip


\end{abstract}

\newpage

\section*{Introduction}

Skew lattices have been studied for the past thirty years as a noncommutative variation of the variety of Lattices with motivations in Semigroup Theory, in Linear Algebra as well as in Universal Algebra. The study of noncommutative lattices begins in 1949 by Pascual Jordan \cite{Jo49} that in 1961 presents a wide review on the subject \cite{Jo61}. It is later approached by Slav\'ik \cite{Sl73a} and Cornish \cite{Co80} that refer to a special variety of noncommutative lattices, namely \textit{skew lattices}. A more general version of these skew lattices is due to Jonathan Leech, first announced in \cite{Le86}. Slavik's algebras lead to a left version of Leech's skew lattices researched by Cornish. The Green's relation $\DD $ defined in a skew lattice \textbf{S} is a congruence and has revealed its important role on the study of these algebras, permitting us a further approach to the coset structure of a skew lattice \cite{Le93}.\medskip

The first section of this paper is dedicated to the approach to skew lattices as algebraic theories discussing several characterizations related with the choice of the possible absorption laws. From it we derive the algebraic category of skew lattices, unveiling several results about these algebras due to the nature of the category they constitute. Decomposition theorems, first presented by Leech in \cite{Le89} and \cite{Le93} are discussed in the third section.  Here we will give special attention to the important role of the natural congruence that in a skew lattice coincides with Green's relation $\DD $. It will determine an adjunction relating skew lattices and lattices and will provide the possibility for the approach developed in the last section. It is known as \textit{Leech's first decomposition}. An alternative decomposition, also mentioned in \cite{Le93} as the complement of \textit{the first decomposition}, regards maximal connected subalgebras instead of maximal rectangular subalgebras. It opens the question of weather it is possible to approach it in a similar way. \textit{Leech's second decomposition} ends the section, referring to a certain "horizontal duality" regarding right/left versions of these algebras, referenced in \cite{Le96} and common to semigroup theory. It is distinct from the "vertical duality", the duality referring to the $\wedge $ and $\vee $ operations, extensively studied in lattice theory. The last section focus on the coset structure of skew lattices, the structural approach provided by the natural congruence $\DD $ giving us many advantages on the further study of skew lattices. The coset structure, introduced for the first time in \cite{Le93}, is revisited here with a categorical approach the follows through all the paper.  \medskip

As this paper aims to present a review on skew lattices as algebraic categories and several of their main properties through the language of Category Theory and Universal Algebra, the reader should be familiar to the basics of these theories.  We suggest further readings in category theory, \cite{La98} and \cite{La98}; in universal algebra \cite{Gr79}, in lattice theory \cite{Ba67} and in semigroup theory \cite{Ho76}. As for notation we will use capital letters $A,B,C,.. $ to represent sets, bold capital letters $\textbf{A,B,C,..}$ to represent algebras, latin letters $\A, \B, \C,.. $ to represent algebraic theories and gothic letters $\AA, \BB, \CC ,...$ to represent categories. Letters $\LL$, $\RR$ and $\DD$ will be kept to represent the Green's relations.

\section{The algebraic theory of skew lattices}

Algebraic theories have been studied under Category Theory generalizing many of the results brought to us by Universal Algebra. In this section we will approach skew lattices as algebraic categories describing the theory of skew lattices, and derive some results from it. For the reader that is familiar with these matters please skip the preliminaries on algebraic categories mentioned here. For a more detailed study on algebraic theories in general please read \cite{Bo94}.  \medskip

A general approach to algebraic structures such as groups or lattices characterize these structures by axiomatizations which involve only equations and operations that must be defined everywhere. The nullary operations are regarded as constants. The \textit{signature} $\sum $ is a family of $k$-ary operations $\set{\sum_k}_{k\in \N }$ where the elements of $\sum_0$ are the constants and the terms of $\sum $ are expressions constructed inductively by the following rules:

\begin{itemize}
\item[] variables $x, y, z, ...$ are terms ;
\item[] if $\langle t_1, ..., t_k \rangle $ is a $k$-tuple of terms and $f\in \sum_k$ then $f\langle t_1, ..., t_k \rangle $ is a term.
\end{itemize}

An \textit{algebraic theory} $\textbf{A}=(\sum , \xi)$ is given by a signature $\sum $ and a set $\xi$ of axioms which are equations between terms\footnote{Algebraic theories are also called equational theories or Lawvere theories.}.

Equivalent algebraic theories are algebraic theories that have the same signature and axioms that one can deduct from the other. The \textit{algebraic theory of skew lattices}, denoted here by $\Sl ^T$, is given by the same signature as the signature as for the algebraic category of lattices, $\set{\wedge, \vee}$, and the following axioms:

\begin{itemize}
\item [$S_1.$] $x\wedge(y\wedge z)=(x\wedge y)\wedge z$  
\item [$S_2.$] $x\vee (y\vee z)=(x\vee y)\vee z$
\item [$S_3.$] $(y\wedge x)\vee x=x$
\item [$S_4.$] $x\wedge (x\vee y)=x$
\item [$S_5.$] $(y\vee x)\wedge x =x$
\item [$S_6.$] $x\vee (x\wedge y)=x$
\end{itemize}

$S_1$ and $S_2$ express associativity which brings independence of order to the operations while the absorption laws $S_2$ to $S_6$ describe the way how both operations relate to each other. Idempotency follows from these axioms: $x\wedge x = x\wedge (x\vee (x\wedge y)) =x = ((y\vee x)\wedge x)\vee x=x\vee x$ and, similarly, $x\vee x=x$ \cite{Le89}. Hence, $S_1$ to $S_6$ are enough to define the theory of skew lattices, $\Sl^T $. Non of these axioms express the existence of a constant and both of the operations are defined everywhere. The next results show us the important role in the choice of the absorption laws. If we have chosen all the absorption laws we would get commutativity and, therefore, lattices. \smallskip

Lattices are skew lattices that satisfy the commutativity axiom for both operations \cite{Le89}, that is,

\begin{itemize}
\item [$S_7.$] $x\wedge y\approx y\wedge x$  
\item [$S_8.$] $x\vee y\approx y\vee x$
\end{itemize}

The \emph{center of a skew lattice} \textbf{S} is defined by the set $Z(S)=\set{x\in X: x\wedge y \approx y\wedge x \text{  and  } x\vee y \approx  y\vee x}$. \smallskip

In \cite{Le02} Leech defines the \emph{algebraic theory of skew* lattices} with the same signature as the algebraic theory of skew lattices, axioms $S_{1}$ and $S_{2}$ and the following others:

\begin{itemize}
\item [$S_9.$] $(x\wedge y)\vee x \approx x$
\item [$S_{10}.$] $x\wedge (y\vee x)\approx x$
\item [$S_{11}.$] $(x\vee y)\wedge x \approx x$
\item [$S_{12}.$] $x\vee (y\wedge x)\approx x$
\end{itemize} 

If we have chosen all the absorption laws we would get commutativity and, therefore, lattices. 

\begin{proposition}\cite{Le02}\label{prop_s23}

The algebraic theory of skew lattices enriched with the axioms $S_{9}$, $S_{10}$, $S_{11}$ and $S_{12}$ is equivalent to the algebraic theory of lattices.

\end{proposition}

The duality principle for skew lattices follows from the well known duality principle for lattices and was referred by Jordan in \cite{Jo61} and later by Slav\'ik in \cite{Sl73a}. According to Slav\'ik, the \textit{dual term} to a term $t$ is defined by the following two rules:

\begin{itemize}
\item[1.] For all variables $x$, $D(x)=x$ ;
\item[2.] If $t_1$ and $t_2$ are terms, $D(t_1\clubsuit t_2)=D(t_1)\clubsuit ' D(t_2)$ with $\clubsuit \neq \clubsuit '$ and $\clubsuit, \clubsuit ' \in \set{\wedge, \vee}$.
\end{itemize}

For an arbitrary formula $\phi $, its \textit{dual formula} $D(\phi )$ is obtained from $\phi $ in such a way that each term occurring in $\phi $ is replaced by its dual term. The \textit{dual theory} $D(T)$ of an algebraic theory $T$ is the set of all $D(\phi )$ where $\phi $ is an element of $T$. A theory is said to be \textit{self dual} iff $D(T)=T$. That is the case of the algebraic theory of skew lattices $\Sl^T $ as well as the algebraic theory of lattices $\La^T $.

\begin{theorem}\cite{Sl73a}\label{duality}

Let $\Gamma $ be a self dual algebraic theory. Then a formula $\phi $ is a consequence of the theory $\Sl ^T$ iff the formula $D(\phi )$ is a consequence of $\Sl ^ T$.

\end{theorem}

This concept of duality was later developed by Leech in \cite{Le89} and \cite{Le02}, where he defined three dual algebras for a skew lattice $(S,\vee, \wedge)$: the \emph{horizontal dual}, $(S,\vee,\wedge )^{h}\approx(S,\vee^{h}, \wedge^{h})$ where $x\vee^{h} y\approx y\vee x$ and $x\wedge^{h} y\approx y\wedge x$; the \emph{vertical dual}, $(S,\vee,\wedge )^{v}\approx(S,\vee^{v}, \wedge^{v})$ where $x\vee^{v} y\approx x\wedge y$ and $x\wedge^{v} y\approx x\vee y$; and the \emph{double dual}, $(S,\vee,\wedge )^{d}\approx(S,\vee^{d}, \wedge^{d})$, given as $(S,\vee,\wedge )^{hv}$. Skew* lattices referenced in \cite{Le02} refer to the horizontal dual. Most of the varieties of skew lattices are closed under one if not all three dualizations.  \medskip

A skew lattice is said to be \emph{symmetric} if it satisfies 

$$x\wedge y \approx y\wedge x \text{   iff   } x\vee y\approx y\vee x.$$

The defining axioms for the \emph{theory of symmetric skew lattices} are $S_1$ to $S_6$ plus the following:

\begin{itemize}
\item [$S_{13}.$] $x\wedge y \wedge (x\vee y\vee x)\approx (x\vee y\vee x)\wedge y\wedge x$ 
\item [$S_{14}.$] $x\vee y \vee (x\wedge y\wedge x)\approx (x\wedge y\wedge x)\vee y\vee x$
\end{itemize}

Skew lattices that satisfy $S_{13}$ are called \emph{lower symmetric skew lattices} and skew lattices that satisfy $S_{14}$ are called \emph{upper symmetric skew lattices}. \medskip

The theory of \emph{right-handed} skew lattices is given by adding to the algebraic theory of skew lattices the following two axioms 

\begin{itemize}
\item [$S_{15}.$] $x\wedge y\wedge x \approx  y\wedge x$ 
\item [$S_{16}.$] $x\vee y\vee x \approx  x\vee y$
\end{itemize}

The theory of \emph{left-handed} skew lattices is defined dually by the following axioms 

\begin{itemize}
\item [$S_{17}.$] $x\wedge y\wedge x\approx  x\wedge y$ 
\item [$S_{18}.$] $x\vee y\vee x \approx  y\vee x$
\end{itemize}

Right-handed skew lattices satisfy $S_9$ and $S_{10}$ while left-handed skew lattices satisfy $S_{11}$ and $S_{12}$. Moreover, 

\begin{proposition}\cite{Le89} The theory of skew lattices that are simultaneously right-handed and left-handed is equivalent to the theory of lattices. \end{proposition} 

In fact, these one-sided versions of a skew lattice introduce another notion of duality distinct to the duality stated in Theorem \ref{duality}. It will be further discussed in the end of the next section.\smallskip

As it happens in the theory of lattices, (the generalization of) the property of distributivity is going to have an important role in the study of the theory of skew lattices. In the next paragraphs we are dedicating some of our attention to it.\smallskip

The defining axioms for the \emph{theory of distributive skew lattices}, as established by Leech in \cite{Le89} are $S_1$ to $S_6$ plus the following:

\begin{itemize}
\item [$S_{19}.$] $x\wedge(y\vee z)\wedge x\approx (x\wedge y\wedge x)\vee (x\wedge z\wedge x)$ 
\item [$S_{20}.$] $x\vee (y\wedge z)\vee x\approx (x\vee y\vee x)\wedge (x\vee z\vee x)$
\end{itemize}

Spinks named these algebras as \textit{middle distributive skew lattices} and confirmed the independency of the axioms $S_{19}$ and $S_{20}$ in \cite{Sp00} presenting a nine-element counter-example obtained by the program SEM, a system for enumerating finite models. Moreover, he showed that the middle distributivity identities are equivalent in the presence of symmetry. Later in \cite{Ka06}, Cvetko-Vah gave a non computational proof of this same equivalence, stating that:

\begin{proposition}\cite{Ka06}
For any skew lattice \textbf{S}, the identities $S_{13}$ and $S_{19}$ imply $S_{20}$; and the identities $S_{14}$ and $S_{20}$ imply $S_{19}$.
\end{proposition}

In order to explore a bit more the different concepts of distributivity available in the literature, consider the following axioms:

\begin{itemize}
\item [$S_{21}.$] $x\wedge(y\vee z)\approx (x\wedge y)\vee (x\wedge z)$ 
\item [$S_{22}.$] $(x\vee y)\wedge z\approx (x\wedge z)\vee (y\wedge z)$
\item [$S_{23}.$] $x\vee(y\wedge z)\approx (x\vee y)\wedge (x\vee z)$ 
\item [$S_{24}.$] $(x\wedge y)\vee z\approx (x\vee z)\wedge (y\vee z)$
\end{itemize}

Observe that in a right handed skew lattice, distributivity reduces to satisfying $S_{22}$ and $S_{23}$ while, in a left handed skew lattice this same property reduces to satisfying $S_{21}$ and $S_{24}$.
Skew lattices satisfying $S_{21}$ and $S_{22}$ are called $\wedge $\emph{-distributive} while skew lattices satisfying $S_{23}$ and $S_{24}$ are called $\vee $\emph{-distributive}. 
\emph{Bidistributivity} is determined by the axioms $S_1$ to $S_6$ and $S_{21}$ to $S_{24}$. Skew lattices satisfying $S_{21}$ and $S_{22}$ are named $\wedge $\emph{-distributive skew lattices} while skew lattices satisfying $S_{23}$ and $S_{24}$ are named $\vee $\emph{-distributive skew lattices}. Note that either of the middle distributivity identities together with the axioms $S_{15}$ to $S_{18}$ imply the axioms of bidistributivity. 
Distributivity in skew lattices was further studied in \cite{Le92}, \cite{Ka08} and \cite{Le11}.

\begin{proposition}\cite{Le11}\label{meetdist}
The algebraic theory of left-handed skew lattices satisfying $S_{19}$ is equivalent to the theory with the same signature defined by the axioms $S_1$ to $S_6$, $S_{17}$ and $S_{21}$. Analogously, the axioms $S_1$ to $S_6$, $S_{15}$ and $S_{22}$ characterize the algebraic theory of right-handed skew lattices satisfying $S_{19}$. Dually, the axioms $S_1$ to $S_6$, $S_{18}$ and $S_{23}$ determine the algebraic theory of left-handed skew lattices satisfying $S_{20}$ and the axioms $S_1$ to $S_6$, $S_{16}$ and $S_{24}$ determine the algebraic theory of right-handed skew lattices satisfying $S_{20}$.
\end{proposition}

\begin{proposition} 
The algebraic theory of left-handed distributive skew lattices is equivalent to the theory with the same signature defined by the axioms $S_1$ to $S_4$, $S_9$, $S_{10}$, $S_{21}$ and $S_{24}$. Analogously, the axioms $S_1$ and $S_2$, $S_{5}$, $S_{6}$, $S_{11}$, $S_{12}$, $S_{22}$ and $S_{23}$ constitute an algebraic theory equivalent to the theory of right-handed distributive skew lattices.
\end{proposition}

\begin{proof}

Let us show that from the axioms  $S_1$ to $S_6$, $S_{17}$, $S_{18}$, $S_{19}$ and $S_{20}$ we can deduct the axioms $S_1$ to $S_4$, $S_9$, $S_{10}$, $S_{21}$ and $S_{24}$. 

$S_{21}$ and $S_{24}$ are directly derived from $S_{17}$, $S_{18}$, $S_{19}$ and $S_{20}$:  $x\wedge(y\vee z)\approx x\wedge(y\vee z)\wedge x\approx (x\wedge y\wedge x)\vee (x\wedge z\wedge x)\approx (x\wedge y)\vee (x\wedge z)$ and  $(y\wedge z)\vee x\approx x\vee (y\wedge z)\vee x\approx (x\vee y\vee x)\wedge (x\vee z\vee x)\approx (y\vee x)\wedge (z\vee x)$.

$S_{9}$ is derivable from $S_{17}$ and $S_{3}$, respectively: $(x\wedge y)\vee x\approx(x\wedge y\wedge x)\vee x\approx x$. Similarly, $S_{10}$ is derivable from $S_{18}$ and $S_{4}$, respectively: $x\wedge (y\vee x)\approx x\wedge (x\vee y\vee x)\approx x$.

Conversely, assume the axioms $S_1$ to $S_4$, $S_9$, $S_{10}$, $S_{21}$ and $S_{24}$. 

$S_{9}$ and $S_{4}$, respectively, are enough to deduct $S_{17}$: $x\wedge y\wedge x\approx x\wedge y\wedge ((x\wedge y)\vee x)\approx x\wedge y$. 
Analogously, $S_{10}$ and $S_{3}$ are enough to deduct $S_{18}$.  

$S_{17}$ and $S_{21}$, respectively, are enough to deduct $S_{19}$: $x\wedge (y\vee z)\wedge x\approx x\wedge (y\vee z) \approx (x\wedge y)\vee (x\wedge z)\approx(x\wedge y\wedge x)\vee (x\wedge z\wedge x)$. 
Analogously, $S_{18}$ and $S_{22}$ are enough to deduct $S_{20}$.   

$S_{5}$ is derived from idempotency, $S_{24}$ and $S_{3}$, respectively: $(y\vee x)\wedge x\approx (y\vee x)\wedge (x\vee x)\approx (y\wedge x)\vee x\approx x$.
Analogously, $S_{21}$ and $S_{4}$ are enough to deduct $S_{6}$.

The second part of the result is proved similarly.  

\end{proof}

Another relevant property, extensively studied in skew lattices is \textit{normality}. Consider the following axioms:

\begin{itemize}
\item [$S_{25}.$] $x\wedge y\wedge z\wedge w\approx x\wedge z\wedge y\wedge w$ 
\item [$S_{26}.$] $x\vee y\vee z\vee w\approx x\vee z\vee y\vee w$
\end{itemize}
\smallskip

A skew lattice $\mathbf S$ is said to be \emph{normal} if it satisfies $S_{25}$. Dually, skew lattices that satisfy $S_{26}$ are named \emph{conormal}. The center of a skew lattice is always a normal skew lattice \cite{Le92}. Moreover, all sub lattices of a skew lattice are normal skew lattices. Normal skew lattices were studied in \cite{Le92} and are sometimes cited as \textit{local lattices} \cite{Le93} or as \textit{mid commutative skew lattices} \cite{Le92}. \smallskip

A skew Boolean algebra is an algebra $\textbf{S }= (S ; \vee, \wedge, \backslash, 0)$ of type $<2, 2, 2, 0>$ such that $(S ; \vee, \wedge, 0)$ is a distributive, normal, symmetric skew lattice with $0$, and $\backslash $ is a binary operation on $\textbf{S}$ satisfying $(x \wedge y \wedge x) \vee (x\backslash y) = x$ and $(x \wedge y \wedge x) \wedge (x\backslash y) = 0$. These algebras, together with skew lattices in unitary rings are the most studied examples of skew lattices for over the past 20 years. Skew Boolean algebras form a variety of algebras \cite{Le90}.  \smallskip

When Pascual Jordan introduced skew lattices in \cite{Jo49}, he chose the axioms $S_1$, $S_2$, $S_9$ and $S_{10}$ that brought him to a weaker version of the algebra that here we present as skew lattice, still holding idempotency \cite{Jo61}. Slav\'ik in \cite{Sl73a} used $S_1$ to $S_4$, $S_9$ and $S_{10}$  and Cornish followed this work using the $\wedge $-bidistributivity axioms $S_{21}$ and $S_{24}$ (cf. \cite{Co80} , Theor 3.4) in order to define Boolean skew algebra, a Boolean version of skew lattice. These brought him to the left-handed version of the skew Boolean algebra that later Leech would state in \cite{Le90}. Leech's skew lattices, firstly presented in \cite{Le89}, are the ones we study through this paper. 

The later approach to algebraic theories is rather syntactic by nature so we'll step towards a representation-free notion of algebraic theories using the language of  category theory, based on the earlier notion of algebraic theory of a skew lattice. We can define the \textit{algebraic category of skew lattices}, denoted by $\Sl $ as follows: 

\begin{itemize}
\item[] as \textit{objects} we take \textit{contexts}, that is, sequences of variables $[x_1,x_2,...,x_n]$, for $n\geq 0$;

\item[] A \textit{morphism} from $[x_1,x_2,...,x_m]$ to $[x_1,x_2,...,x_n]$ is an $n$-tuple $\langle t_1,t_2,...,t_n \rangle$, where each $t_k$ is a term of the theory whose variables are among $x_1,...,x_m$. Every such term is built inductively. The equality of two such morphisms $\langle t_1,...,t_n \rangle$ and $\langle u_1,...,u_n \rangle$ holds exactly when the axioms $S_1$ to $S_6$ imply $t_k=u_k$ for each $k=1,...,n$. In other words, morphisms are equivalence classes of terms, where two terms are equivalent when the theory proves them to be equal. 

\item[] The \textit{composition of morphisms} $\langle t_1,...,t_m \rangle :[x_1,...,x_k]\rightarrow [x_1,...,x_m]$ and $\langle u_1,...,u_n \rangle :[x_1,...,x_m]\rightarrow [x_1,...,x_n]$ is the morphism $\langle v_1,...,v_n \rangle $ whose $i$-th component is obtained by simultaneously substituting in $u_i$ the terms $t_1,...,t_m$ for the variables $x_1,...,x_m$: $v_i=u_i[t_1,...,t_m/x_1,...,x_m]$, for $1\leq i\leq n$.

\item[] The \textit{identity morphism} on $[x_1,...,x_n]$ is $\langle x_1,...,x_n \rangle $. 
\end{itemize}\medskip

The object $[x_1,..., x_{n+m}]$ is the product of $[x_1,...,x_n]$ and $[x_1,...,x_m]$ so that all finite products exist. Furthermore, every object is a product of finitely many instances of the object $[x_1]$. \textit{Models} of algebraic theories are just finite product preserving functors \cite{Bo94}. A finite product preserving functor is then determined, up to natural isomorphism, by its action on the context $[x_1]$ and the terms representing the basic operations $\vee $ and $\wedge $. This suggests that the category of the models of the algebraic theory $\Sl^T $ in $Set$ is equivalent to the algebraic category of the skew lattices, $\Sl $, provided both categories have the same notion of morphisms. In some sense algebraic categories are the Category Theory approach to the well know Universal Algebra concept of algebraic variety, a generalization of this \textit{finitary} notion. On what follows we will also use the notation $\La $ to represent \textit{the algebraic category of lattices} that is well known and has a similar construction.\smallskip

All algebraic categories are regular categories \cite{Bo94}. Moreover, this algebraic category approach to skew lattices provide us with general results valid in this context as follows:  

\begin{theorem} (Morphism Factorization Theorem)\label{homofact}
In the algebraic category of skew lattices, every morphism $f:A\rightarrow B$ factors as a composition of a regular epimorphism $q$ and a monomorphism $m$. Moreover, the factorization is unique up to isomorphism.
\end{theorem} 

This and other results following from the fact that skew lattices can be seen as algebraic categories reveal the importance of this approach.

\section{The decomposition theorems}

Green's relations $\LL $, $\RR $ and $\DD $ have proved to have an important role in the further development of skew lattice theory. On the following section we approach $\DD $ within a categorical perspective and emphasize its important role on the study of these algebras. This discussion will follow argumenting the relevance of its contribution due to the fact that it determines such a partition on the skew lattice that each of its blocks are its maximal "anticommutative" subalgebras. \medskip

Recall the well known \textit{forgetful functor} $\UU$ between the algebraic category of abelian groups $\G_{ab}$ and the algebraic category of groups $\G $ that forgets the axiom of commutativity and brings all abelian groups to the wider algebraic category of groups. It can also be defined a functor $\AA $, the \textit{abelianization functor}, that will associate to any group $G$ an abelian group, $G/[G,G]$ while morphisms are not affected by commutators. This example of adjoint functors $\AA \dashv \UU $ gives us an intuitive idea of how we can approach the problem of relating the algebraic categories $\La $ and $\Sl $. \medskip

For the reader familiar with categorical treatment of algebraic theories the following result is more of an observation than a real theorem. Even though not being original is quite relevant within the context of this paper and shall be stated as a theorem. 

\begin{theorem}\label{te_fundamental}

Let $\A $ and $\A' $ be algebraic structures such that $ \xi_{\A'} = \xi_{\A} \cup E,$ where $\xi _{\A}$ is the set of axioms of the algebraic structure $\A$. Then the forgetful functor $\UU :\A ' \longrightarrow \A $ has a left adjoint: the functor  $\FF :\A \longrightarrow \A '$ that assigns to each algebra $\textbf{A}$ in $\A $ the quotient algebra $\textbf{A}/\sim $, where $\sim $ is the congruence determined by the set of axioms $E$. 

\end{theorem}

\begin{proof}

Consider an algebraic theory $\A$ and another algebraic theory $\A'$ that is built with the signature of $\A$ plus some new identities $E$, ie, $$ \xi_{\A'} = \xi_{\A} \cup E,$$ where $\xi _{\A}$ is the set of axioms of the algebraic structure $\A$. The adjoint functor theorem makes it quite clear that the forgetful functor $\UU :\A' \longrightarrow \A$ has a left adjoint, but we will describe this in detail. Let $\textbf{A}$ be an algebra in $\A $. As the lattice of all congruences in $\textbf{A} $, $(Con_{ A} , \cap ,\cup ) $ is a complete lattice for the inclusion relation, consider $\sim $ to be the least congruence containing all equations of $\A'$. Now consider the functor $\AA :\A \longrightarrow \A'$ that assigns to each algebra $\textbf{A}$ in $\A $ the quotient $\textbf{A}/ \sim $ in $\A'$ and to each morphism in $\A $ the same morphism in $\A' $. 

As morphisms are not about satisfying equations but preserving operations, the image of the morphisms in $\A$ are not affected by the congruence $\sim $ and remain morphisms. Observe that if $\textbf{A}$ is an algebra of $\A $ then the quotient $\textbf{A}/ \sim $ is the smallest image of $A$ that respects all the axioms of $\A' $ due to the fact that $\sim $ is the smallest congruence containing all equations of $\A'$. \medskip

$\eta_{A}: \textbf{A} \rightarrow \AA \UU (\textbf{A})$ is the unit of the left adjunction $\AA \dashv \UU$ and the commutativity of the following diagram shows that $\eta_{A}$ has the universal property: Let $\textbf{A}\in \A $, $\textbf{B}\in \A'$ and $g : \textbf{A} \rightarrow \AA (\textbf{B}) $. Then, exists a unique $f: \textbf{A}/\sim \rightarrow \textbf{B}$ such that $g=\AA (f)\circ \eta_{A}$.

\begin{center}

\begin{tikzpicture}
\path (-1,2) node[] (G) {$\textbf{A}$};

\path (2,2) node[] (GG) {$\AA (\textbf{A} /\sim )$};

\path (2,-1) node[] (H) {$\AA (\textbf{B} )$};

\path (4,2) node[] (GGG) {$\textbf{A} /\sim $};
\path (4,-1) node[] (HH) {$\textbf{B} $};

\draw[arrows=-latex'] (GGG) -- (HH) node[pos=.5,left] {$\exists ! f$};

\draw[arrows=-latex'] (G) -- (GG) node[pos=.7,above] {$\eta_{\A}$};
\draw[arrows=-latex'] (GG) -- (H) node[pos=.5,right] {$\AA (f)$};
\draw[arrows=-latex'](G) -- (H) node[pos=.5,left] {$g$};
\end{tikzpicture}
\end{center} 

\end{proof}

When we consider the\textit{ algebraic theory of a skew lattice} and the \textit{algebraic theory of a lattice}, the one identity that enriches the second structure is again commutativity. Consider the \textit{forgetful functor} $\UU : \La \rightarrow \Sl $  and the \textit{"abelianizor" functor} $\AA : \Sl \rightarrow \La $ that, in this context, assigns to each skew lattice $\textbf{S}$ its correspondent commutative version, that is, its correspondent lattice $\textbf{S}/\sim $, where $\sim $ is determined by the commutativity axiom. On the following we will show that, in this case, the congruence $\sim $ determined by the axiom $x\wedge y=y\wedge x$ is the Green's relation $\DD $. Let us start by stating some preliminaries.\medskip

 A \textit{band} is a semigroup of idempotents. A \textit{semilattice} is a commutative band. When $S$ is a commutative semigroup, the set $E(S)$ of all idempotents in $S$ is a semilattice under the semigroup multiplication. When $S$ is not commutative, $E(S)$ needs not to be closed under multiplication. Skew lattices can be seen as double bands, $(S,\wedge )$ and $(S,\vee )$, with an extra property, absorption, that relates the two operations $\wedge $ and $\vee $ in the sense that\smallskip

\begin{center}
$x\vee y=x$ iff $x\wedge y = y$ and $x\vee y = y$ iff $x\wedge y=x$.
\end{center}

Influenced on the natural partial order and the Green's relations defined for bands in \cite{Ho76}, we define in a skew lattice $S$:
\begin{itemize}
\item[] the \textit{natural partial order} by $x\geq y$ iff $x\wedge y=y=y\wedge x$, or dually, $x\vee y = x = y\vee x$;
\item[] the \textit{natural preorder} by $x\succeq y$ iff $x\vee y\vee x = x$  or, dually, $y\wedge x\wedge y = y$;
\item[] the \textit{natural equivalence} by $x\equiv y$ iff $x\preceq y$ and $y\preceq x$.
\end{itemize}

The Green's relations simplified for bands by Howie in \cite{Ho76}, can be defined in a band $S$ by:

\begin{itemize}
\item[] $x\RR y \text{  iff  }  xy=y \text{ and } yx=x$ 
\item[] $x\LL y \text{  iff  } xy=x \text{ and } yx=y$ 
\item[] $x\DD y \text{  iff  } xyx=x \text{ and } yxy=y$ 
\end{itemize}
\medskip

$\DD $ is a congruence relation on any band $S$. $\LL $ and $\RR $ need not to be. Moreover, $\DD = \RR \circ \LL = \LL \circ \RR = \RR \vee \LL$ and $\RR \wedge \LL = \RR \cap \LL := \HH $.
\medskip

In a skew lattice, Leech \cite{Le96} defines the Green's relations as follows:

\begin{itemize}
\item[] $\RR=\RR_{\wedge}=\LL_{\vee}$  
\item[] $\LL=\LL _{\wedge}=\RR_{\vee}$ 
\item[] $\DD =\DD _{\wedge}=\DD_{\vee}$ 
\end{itemize}
\medskip

All of these relations are congruences on any skew lattice \textbf{S}. Right-handed skew lattices are the skew lattices for which $\RR =\DD$ while left-handed skew lattices are determined by $\LL =\DD$ \cite{Le96}.\medskip

As all elements $\leq $-related are $\preceq $-related, the natural preorder $\preceq $ is \emph{admissible} with respect to the natural partial order $\leq $. The fact that the $\DD$ equivalence can be expressed by the preorder allows us to draw diagrams like the one on the Figure \ref{fig:skewdiag}, that are capable to represent skew lattices: $a$ and $b$ are $\DD$ related as expressed by the dashed segments and all others are related by the natural partial order as expressed by full segments. \medskip

\begin{figure}  
\begin{center}  
\begin{pspicture}(-2,-2)(2,2)

\psline[linewidth=0.5 pt,linestyle=dashed]{*-*}(-1.5,0)(-0.5,0) 
\psline[linewidth=0.5 pt,linestyle=dashed]{*-*}(0.5,0)(1.5,0)
\psline[linewidth=0.5 pt,linestyle=dashed]{*-*}(-0.25,1)(0.25,1)

\psline[linewidth=0.5pt]{*-*}(-1.5,0)(-0.25,1) 
\psline[linewidth=0.5 pt]{*-*}(-0.5,0)(0.25,1) 
\psline[linewidth=0.5pt]{*-*}(0.5,0)(-0.25,1) 
\psline[linewidth=0.5 pt]{*-*}(1.5,0)(0.25,1)

\psline[linewidth=0.5pt]{*-*}(0.5,0)(0,-1)
\psline[linewidth=0.5 pt](1.5,0)(0,-1) 
\psline[linewidth=0.5 pt](-0.5,0)(0,-1) 
\psline[linewidth=0.5 pt](-1.5,0)(0,-1) 

\uput[140](-0.5,0){$b$} 
\uput[140](-1.5,0){$a$} 
\uput[40](0.5,0){$c$}
\uput[40](1.5,0){$d$}  
\uput[40](0.25,1){$2$} 
\uput[140](-0.25,1){$1$} 
\uput[270](0,-1){$0$} 

\end{pspicture}

\caption{\small \sl The diagram of a skew lattice.} \label{fig:skewdiag} 

\end{center}  
\end{figure}
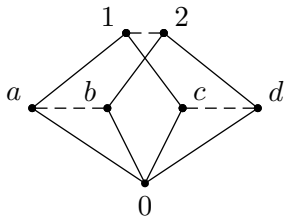 

\begin{theorem}\cite{Le89} \label{comsup}
The center of a skew lattice $\textbf{S}$ is the subalgebra formed by the union of all its singleton $D$-classes. In particular, $\textbf{S}$ is a lattice if either $\vee $ or $\wedge $ is commutative. 
\end{theorem}

Given nonempty sets, $L$ and $R$, its product $L\times R$ is a skew lattice under the operations $(x,y)\vee (x',y') = (x',y)$ and $(x,y)\wedge (x',y') = (x,y')$.  A \textit{rectangular} skew lattice is a isomorphic copy of this skew lattice. All $\DD $-equivalence classes are rectangular, both in bands and in skew lattices. When working with bands in \cite{Mc54}, McLean referred to these rectangular algebras as \textit{anticommutative idempotent semigroups}, that is, bands for which no two distinct elements commute. According to Proposition \ref{order_dclass} as well no two elements in each $\DD $-class are order related.\medskip

Moreover, he characterized these by the identity $abc=ac$ and proved that, in a band $\textbf{S}$, there exists a homomorphism $\phi $ of $\textbf{S}$ onto a semilattice $\textbf{T}$ such that the inverse image of any element of $\textbf{T}$ is a band and $\phi $ is the weakest in the sense that any other commutative homomorphic image of $\textbf{S}$ is also a homomorphic image of $\textbf{T}$ . In other words, the congruence classes of $\DD $ form maximal rectangular subbands of $\textbf{S}$ and the quotient algebra $\textbf{S}/\DD $ is the maximal semilattice image of $\textbf{S}$. Thus, we can look at a band as a semilattice diagram with each node filled in by a rectangular band. \medskip

Furthermore, influenced by the Clifford-McLean Theorem for bands \footnote{Result independently by A. H. Clifford in \cite{Cl41} thus known as Clifford-McLean theorem for bands.}, Leech stated in \cite{Le89} the following result known as the Leech's first decomposition theorem: 

\begin{theorem}\cite{Le89}\label{1decomp}
Let $\textbf{S}$ be a skew lattice. Then, $\DD$ is a congruence in $\textbf{S}$, $\textbf{S}/\DD $ is the maximal lattice image of $\textbf{S}$ and all congruence classes of $\DD $ are maximal rectangular skew lattices in $\textbf{S}$. The maximal rectangular subalgebras of a skew lattice form a partition with the induced quotient algebra being the maximal lattice image of the given skew lattice.
\end{theorem} 

The natural equivalence $\DD $ is the "key" to look at this \textit{First Decomposition Theorem} through the general result stated in Theorem \ref{te_fundamental}. In the blocks of $\mathbf{S}/\DD$ we collapse such maximal sets where no two elements commute.  
The commutativity of one of the operations in a skew lattice gives us the commutativity of the other operation. Considering the new equation to be the commutativity of one of them, $\Sl $ is the algebraic category corresponding to $\A $ and $\La $ is the algebraic category corresponding to $\A '$. The functor $\FF : \Sl \rightarrow \La $ is left adjoint to  $\UU : \La \rightarrow \Sl $. Hence, by the uniqueness of the adjunction, the congruences $\sim $ and $\DD $ coincide, up to isomorphism. Moreover, $\DD $ is the smallest congruence for which the quotient $\textbf{S}/\DD $ satisfies the property of commutativity. \medskip 

On the following we present a characterization for the natural order by an identity, derive a characterization for right(left)-handed skew lattices and show that all $\DD $ classes are composed of unrelated elements (with respect to the order relation). 

\begin{lemma}\label{order_id}

Let $\textbf{S}$ be a skew lattice and $x,y\in S$. Then $x\geq y$ iff $y=x\wedge y\wedge x$ or, dually, $x=y\vee x\vee y$. 

\end{lemma}

\begin{proof} 
 
 Let $x,y\in S$. If $x\geq y$ then $x\wedge y\wedge x=x\wedge x=x$. 
 
Conversely, $y=x\wedge y\wedge x\geq x$.
 
\end{proof}

\begin{proposition}\label{order_sides}

Let $\textbf{S}$ be a skew lattice. $\textbf{S}$ is right-handed iff for all $x,y\in S$, $y\wedge x\leq x$ and $x\leq x\vee y$. Analogously, $\textbf{S}$ is left-handed iff for all $x,y\in S$, $x\wedge y\leq x$ and $x\leq y\vee x$.

\end{proposition}

\begin{proof} 
 
Let $x,y\in S$. By Lemma \ref{order_id}, $y\wedge x=x\wedge y\wedge x$ is equivalent to $y\wedge x\leq x$ as well as $x\vee y= x\vee y\vee x$ is equivalent to $x\leq x\vee y$. The left-handed case is analogous.
 
\end{proof}

\begin{proposition}\cite{Le89}\label{order_dclass}

Let $\textbf{S}$ be a skew lattice and $x,y\in S$. Then $x\geq y$ and $x\DD y$ implies $y=x$.

\end{proposition}

Proposition \ref{order_dclass} unveils that, in a skew lattice, all $\DD $-classes constitute antichains. Hence, all the order structure of a skew lattice is unveiled by the coset bijections.\medskip

Can one describe a skew lattice through its partial order? In the case of lattices its known that $x\leq y \Leftrightarrow x\wedge y=x $ determines the isomorphism between he order structure and the algebraic structure. In the context of skew lattices Lemma \ref{order_id} expresses a similar equivalence but its not able to fully describe the operations. Unlike what happens for lattices, a skew lattice is not determined by its (partial) order structure. For instance, the case of the diagrams on the figure \ref{fig_det} show an Hasse diagram corresponding to two different skew lattices. Recall that, in the case of lattices, the Hasse diagram determines the order structure of the lattice. \medskip

\begin{figure}

\begin{center}

$\begin{array}{ccc}

\begin{pspicture}(-2,-2)(2,2) 

\psline[linewidth=0.5pt]{*-*}(-1,0)(0,1) 
\psline[linewidth=0.5pt]{*-*}(1,0)(0,1) 

\psline[linewidth=0.5pt]{*-*}(1,0)(0,-1)
\psline[linewidth=0.5 pt](-1,0)(0,-1) 

\uput[180](-1,0){$a$} 
\uput[1](1,0){$b$} 
\uput[90](0,1){$1$} 
\uput[270](0,-1){$0$} 

\end{pspicture}

&

&

\begin{pspicture}(-2,-2)(2,2) 

\psline[linewidth=0.5pt]{*-*}(-1,0)(0,1) 
\psline[linewidth=0.5pt]{*-*}(1,0)(0,1) 
\psline[linewidth=0.5 pt,linestyle=dashed]{*-*}(-1,0)(1,0)
\psline[linewidth=0.5pt]{*-*}(1,0)(0,-1)
\psline[linewidth=0.5 pt](-1,0)(0,-1) 

\uput[180](-1,0){$a$} 
\uput[1](1,0){$b$} 
\uput[90](0,1){$1$} 
\uput[270](0,-1){$0$}  

\end{pspicture}

\\

S_1

&

&

S_2

\end{array}$

\end{center}

\caption{Two skew lattices with the same order structure.}\label{fig_det}

\end{figure}
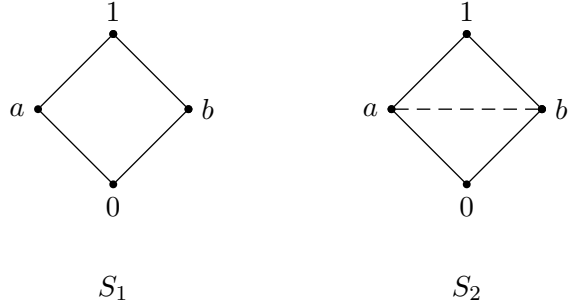

All posets determine categories. Posets are equivalent to one another if and only if they are isomorphic. According to Proposition \ref{order_dclass}, that is the case of the poset category of a skew lattice $\textbf{S}$ and the poset category of the lattice that is determined by the correspondent quotient $\textbf{S}/\DD $. In \cite{Le89}, Leech defines the \textit{natural graph} of a skew lattice $S$ as the undirected graph $(S,E)$ given by the natural partial order of $S$, where $S$ is the set of vertices and $\set{x,y}$ forms an edge in $E$ whenever $x>y$ or $y>x$. The natural graph of a skew lattice \textbf{S} is a part of the \textit{Hasse diagram of the skew lattice} and equivalent to the poset category of \textbf{S}. 
\medskip

A skew lattice is \textit{connected} when its natural graph is connected, that is, when for every pair of vertices $(x,y)$, the graph contains a path from $x$ to $y$. A \textit{component} of $S$ is any connected component of its graph. Already here we can see the intersection with the $\DD $-congruence: $S/\DD $ is a lattice and therefore its Hasse diagram is a connected graph, thus each component of $S$ has a nonempty intersection with each equivalence class of $S$. The following result stated in \cite{Le89} is the complement to Theorem \ref{1decomp}:

\begin{theorem}\cite{Le93}\label{cdecomp}
The components of a skew lattice $S$ are its maximal connected subalgebras. Moreover, the partition of $S$ into components is a congruence partition for which the induced quotient algebra is the maximal rectangular image of $S$. 
\end{theorem}

Regarding the statement of Theorem \ref{te_fundamental}, its unavoidable to open the question wether or not the congruence in Theorem \ref{cdecomp} is determined by an axiom so that this result is a particular case. \medskip

Skew lattices are (double) regular bands, that is, both operations satisfy the identity $xyxzx=xyzx$. This property is referred by Leech in \cite{Le89} as biregularity. In \cite{Ki58}, Kimura studies the structure of regular bands presenting in this context a factorization theorem stating that, when $\textbf{S}$ is a regular band, there exist a left regular band $\textbf{L}$ and a right regular band $\textbf{R}$ both of which have the same structure semilattice $\textbf{C}$ such that $\textbf{S}$ is isomorphic to the spined product of $\textbf{L}$ and $\textbf{R}$ with respect to $\textbf{C}$. By \textit{spined product} is meant:
 $$\textbf{R} \bowtie \textbf{L} =\set{(x,y): x\in R, y\in L, p(x)=q(y)}$$ when $p:\textbf{L}\rightarrow \textbf{C}$ and $q:\textbf{R}\rightarrow \textbf{C}$ are natural homomorphisms.This was later called \textit{fibered product} and is equivalent to the categorical concept of \textit{equalizer} of $\textbf{L}$ and $\textbf{R}$ with respect to $\textbf{C}$. This result is dependent of regularity, which is natural to skew lattices, and also has an analogue in that context known as Leech's \textit{Second Decomposition Theorem} \cite{Le89},

\begin{theorem}\cite{Le89}\label{2decomp}
The relations $\LL $ and $\RR $ are both congruences on any skew lattice $\textbf{S}$. Moreover, 
\begin{itemize}
\item[(i)] $S/\LL $ is the maximal right-handed image of $\textbf{S}$. 
\item[(ii)] $S/\RR $ is the maximal left-handed image of $\textbf{S}$. 
\item[(iii)] The induced epimorphisms $\textbf{S} \rightarrow \textbf{S}/\LL$ and $\textbf{S} \rightarrow \textbf{S}/\RR$ together yield an isomorphism of $\textbf{S}$ with the fibered product $\textbf{S}/\LL \times_{S/\DD } S/\RR $. 
\end{itemize}
\end{theorem} 

In other words, every skew lattice factors as the fibered product of a right-handed skew lattice with a left-handed skew lattice over $S/\DD $, with both factors being unique up to  isomorphism. In the language of category theory, while considering models of the algebraic theory of skew lattices in the category of sets, this fibered product is the limit of a diagram consisting of two morphisms $f:\textbf{S}/\LL \rightarrow S/\DD $ and $g: \textbf{S}/\RR \rightarrow S/\DD $ with a common codomain $S_{\DD }$, the equalizer of $\textbf{S}/\LL$ and $\textbf{S}/\RR$ over  $S_{\DD }$. This is described on the following proposition:  

\begin{proposition}\cite{Le89}\label{2decomp_cat}

The following diagram is a pullback on the algebraic category of skew lattices. 

\begin{center}
\begin{tikzpicture}
\path (-1.5,1.5) node[] (S) {$S$};

\path (-1.5,-1.3) node[] (s1) {};
\path (-1.3,-1.5) node[] (s2) {};
\path (-1.3,-1.3) node[] (s3) {};

\path (-1.5,-1.5) node[] (L) {$S/\LL$};

\path (1.5,1.5) node[] (R) {$S/\RR$};

\path (1.5,-1.5) node[] (D) {$S/\DD$};

\draw[arrows=-latex'] (s1) -- (s3) node[pos=.5] {};
\draw[arrows=-latex'] (s2) -- (s3) node[pos=.5] {};
\draw[arrows=-latex'] (S) -- (R) node[pos=.5,right] {};
\draw[arrows=-latex'](S) -- (L) node[pos=.5,right] {};
\draw[arrows=-latex'] (L) -- (D) node[pos=.5,right] {};
\draw[arrows=-latex'](R) -- (D) node[pos=.5,right] {};
\end{tikzpicture}
\end{center}

\end{proposition}

As any skew lattice $S$ is embedded in the product $S/\RR \times S/\LL $, joint properties of $S/\RR $ and $S/\LL $ are often passed on to $S$, and conversely. In particular, 

\begin{theorem}\cite{Ka08}\label{duality2}
If, \textbf{S} is a skew lattice, $S/\RR $ and $S/\LL $ belong to a variety if and only if $S$ does.
\end{theorem}

Theorem \ref{duality2} is expressing a duality analogous to the one observed in semigroups and distinct to the one stated in Theorem \ref{duality}. It is rather useful as it can be observed in several proofs of \cite{Ka08} and \cite{Jo10}. 
 
\begin{proposition}\cite{Ka06}\label{duality2b}
A skew lattice \textbf{S} satisfies any identity or equational implication satisfied by both its left factor $S/\RR $ and its right factor $S/\LL $.
\end{proposition}

An example for which Proposition \ref{duality2b} isn't sufficient is proving that \textit{$\circ $ always equals $\nabla $}. In fact, this is untrue for a skew lattice \textbf{S} but true for both its left factor $S/\RR $ and its right factor $S/\LL $. Most of disjunctions of identities fall into this class of examples.

\section{The coset category}

The study of the coset structure of skew lattices began with Leech in \cite{Le93}. It derives from the first decomposition theorem stated in Theorem \ref{1decomp} and gives an introspective into the role of the partition that $\DD$-classes determine on each other providing important additional information. Cosets are irrelevant in both the context of Semigroup Theory or Lattice Theory being something very specific to skew lattices. Though, the coset structure reveals a new perspective that does not have a counterpart either in the theory of lattices or in the theory of bands. Several varieties of skew lattices were characterized using laws involving only cosets in \cite{Jo10}. In this section we revisit the category that describes this internal structure theory, the \textit{coset category}, and we will establish a relation between the coset structure and the poset structure of a skew lattice. \medskip

Consider a skew lattice $\mathbf S$ consisting of exactly two $\DD$-classes
$A>B$. Given $b\in B$ the subset
$A\wedge b\wedge A = \set{a\wedge b\wedge a\,|\,a\in A}$ is said to be a
\emph{coset} of $A$ in $B$. Similarly, a coset of $B$ in $A$ is any
subset
$B\vee a\vee B = \set{b\vee a\vee b\,|\, b\in B}$ of $A$, for a fixed
$a\in A$.  Given $a\in A$, the \emph{image} of $a$ in $B$ the set

 $$a\wedge B\wedge a = \set{a
\wedge b\wedge a\,|\,b\in B}=\set{b\in B\,|\,b\leq a}.$$  Dually, given $b\in B$ the set
$b\vee A\vee b = \set{a\in A:b\leq a}$ is the image of $b$ in $A$.

\begin{proposition}\cite{Jo10}\label{strg_prop}
Let $\mathbf S$ be a skew lattice with comparable $\DD$-classes $X>Y$
and let
$y,y'\in Y$. The following are equivalent:
\begin{itemize}
\item[(i)] $X\wedge y \wedge X = X\wedge y' \wedge X$,
\item[(ii)] for all $x\in X$, $x\wedge y\wedge x = x\wedge y'\wedge
x$,
\item[(iii)] there exists $x\in X$ such that
$x\wedge y\wedge x = x\wedge y'\wedge x$.
\end{itemize}
The dual result also holds.
\end{proposition}

We call to a skew lattice \textbf{S} \textit{primitive} if it is composed by just two comparable $\DD $-classes, \textit{skew chain} when $S/\DD $ is a chain, \textit{diamond} when it is composed by two incomparable $\DD $-classes, $A, B$, a join class $J$ and a meet class $M$.  
\medskip

Due to absorption and regularity, the following result holds unveiling the coset structure of skew lattices through the description of a double partition that cosets induced on each other.

\begin{theorem}\cite{Le93} \label{coset_part}
Let $\mathbf{S}$ be a skew lattice with comparable $\DD$-classes $A>B$. Then, 
\begin{itemize}
\item[i)] $B$ is partitioned by the cosets of $A$ in $B$; dually $A$ is partitioned by the cosets of $B$ in $A$. 
\item[ii)] The image of any element $a\in A$ in $B$ is a transversal of the 
cosets of $A$ in $B$; dual remarks hold for any $b\in B$ and the cosets of $B$ in $A$. 
\item[iii)] Given cosets $B\vee a\vee B$ in $A$ and $A\wedge b\wedge A$ in $B$ a natural bijection of cosets, named \textit{coset bijection} is given by: $x\in B\vee a\vee B$ corresponds to $y\in A\wedge b\wedge A$ if and only if $x\geq y$. 
\item[iv)] The operations $\wedge$ and $\vee$ on $A\cup B$ are determined jointly by the coset bijections and the rectangular structure of each $\DD$-class.
\end{itemize}
\end{theorem}

According to Theorem  \ref{coset_part}, in  a skew lattice $\textbf{S}$ with comparable $\DD$-classes $A>B$ where $a\in A$ and $b\in B$, a coset bijection $\varphi_{a,b}$ from $B\vee a\vee B$ to $A\wedge b\wedge A$ maps an element $x\in B\vee a\vee B$ in $A$ to an element $y\in A\wedge b\wedge A$ in $B$. If $x\in A$ then $x\wedge b\wedge x$ is the only element of $A\wedge b\wedge A$ bellow $x$:  if $x\wedge b\wedge x<x$ and $x\wedge b\wedge x, y\in A\wedge b\wedge A$, Theorem  \ref{coset_part} implies that $x\wedge b\wedge x =y$. Dually, the inverse of this coset bijection sends an element $y\in B$ to an element $y\vee a\vee y\in A$. When we consider a skew lattice $\textbf{S}$ with comparable $\DD$-classes $A>B>C$, a nonempty composition of coset bijections from $B\vee a\vee B$ to $B\wedge c\wedge B$ maps an element $x\in A$ to $(x\wedge b\wedge x)\wedge c\wedge (x\wedge b\wedge x)$ that equals $x\wedge c\wedge x$ by Proposition \ref{sub} in case the composition of these coset bijections is still a coset bijection.\medskip

\begin{proposition}\cite{Ka05}\label{sub}
Let $A>B>C$ denote three distinct but comparable equivalence classes of a skew lattice $S$. Then,
\begin{itemize} 
\item[(i)]  For any $c\in C$, the $A$-coset $A\wedge c\wedge A$ is contained in the $B$-coset $B\wedge c\wedge B$. Likewise, for any $a\in A$, the $C$-coset $C\vee a\vee C$ is contained in the $B$-coset $B\vee a\vee B$;
\item[(ii)] Given $a>b>c$ with $a\in A$, $b\in B$ and $c\in C$, if $\varphi $ is the coset bijection from $A$ to $B$ taking $a$ to $b$, $\psi $ is the coset bijection from $B$ to $C$ taking $b$ to $c$ and finally $\chi $ is the coset bijection from $A$ to $C$ taking $a$ to $c$, then $\psi \circ \varphi \subseteq \chi $.
\end{itemize} 
\end{proposition} 

The second part of this result follows from the inclusion 

$$(A\wedge b\wedge A\bigcap C\vee b\vee C)\wedge c\wedge (A\wedge b\wedge A\bigcap C\vee b\vee C) \subseteq A\wedge b\wedge A\wedge c\wedge A\wedge b\wedge A=A\wedge c\wedge A$$ showing us that the composite partial bijection $\psi \circ \varphi $ if nonempty is part of a coset bijection from a $C$-coset of $A$ to an $A$-coset of $C$. \medskip


\textit{Normality} for bands is characterized by the identity $uxyv=uyxv$. This is equivalent to $eSe$ being a semilattice for all $e\in  S$, or else to  $B$ covering $A$ whenever $A\geq B$ are comparable $\DD $ classes in $S$. By $B$ covering $A$ is meant
\begin{center}
$\forall a\in A$ $\exists ! b\in B$ such that $a\geq b$.
\end{center}\medskip

Recall that a skew lattice is \textit{normal} if $(S, \wedge )$ is a normal band. Normal skew lattices form a subvariety of skew lattices. Any maximal normal band in a ring $R$ is a normal skew lattice under $\nabla $ and the usual multiplication \cite{Ka05}. Hence the relevance of this property. This reveals the coset structure of normal skew lattices. Moreover, coset bijections are closed under composition with the composition of adjacent coset bijections being nonempty, making sure that compositions of coset bijections are coset bijections. \medskip


A skew lattice is \textit{categorical} if nonempty composites of coset bijections are coset bijections. Categorical skew lattices form a variety \cite{Le93}. The study of distributive skew chains has been a relevant motivation to the studies on categorical skew lattices. Skew lattices in rings as well as normal skew lattices are examples of algebras in this variety \cite{Le93}. The example represented by the diagram on the Figure \ref{fig_noncat} shows a skew lattice that need not be categorical.

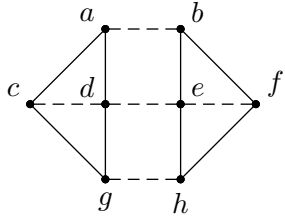
\begin{figure}\label{noncat}
\begin{center}
\begin{pspicture}(-2,-2)(2,2)

\psline[linewidth=0.5 pt,linestyle=dashed]{*-*}(-1.5,0)(-0.5,0) 
\psline[linewidth=0.5 pt,linestyle=dashed]{*-*}(0.5,0)(1.5,0)
\psline[linewidth=0.5 pt,linestyle=dashed]{*-*}(-0.5,1)(0.5,1)
\psline[linewidth=0.5 pt,linestyle=dashed]{*-*}(-0.5,-1)(0.5,-1)
\psline[linewidth=0.5 pt,linestyle=dashed]{*-*}(-0.5,0)(0.5,0)

\psline[linewidth=0.5pt]{*-*}(-1.5,0)(-0.5,1) 
\psline[linewidth=0.5 pt]{*-*}(-0.5,0)(-0.5,1) 
\psline[linewidth=0.5pt]{*-*}(0.5,0)(0.5,1) 
\psline[linewidth=0.5 pt]{*-*}(1.5,0)(0.5,1)

\psline[linewidth=0.5pt]{*-*}(0.5,0)(0.5,-1)
\psline[linewidth=0.5pt]{*-*}(1.5,0)(0.5,-1) 
\psline[linewidth=0.5pt]{*-*}(-0.5,0)(-0.5,-1) 
\psline[linewidth=0.5pt]{*-*}(-1.5,0)(-0.5,-1)

\uput[140](-0.5,0){$d$} 
\uput[140](-1.5,0){$c$} 
\uput[40](0.5,0){$e$}
\uput[40](1.5,0){$f$}  
\uput[40](0.5,1){$b$} 
\uput[140](-0.5,1){$a$} 
\uput[270](-0.5,-1){$g$} 
\uput[270](0.5,-1){$h$}

\end{pspicture}
\end{center}

\caption{A non categorical skew lattice}\label{fig_noncat}

\end{figure}\medskip

The name comes directly from the definition as \textit{categorical skew lattices} are the ones for whom coset bijections form a category under certain conditions. A categorical skew lattice is named \textit{strictly categorical} if compositions of coset bijections are never empty. Already in 1993, Leech defines categorical skew lattices within his geometric perspective on skew lattices \cite{Le93}. For any strictly categorical skew lattice \textbf{S}, define the category $\CC $ by the following:  

\begin{itemize}
\item[] the objects of $\CC$ are the $\DD $-classes of $S$ ;
\item[] for comparable  classes $A > B$, $\CC (A, B)$ are all the coset bijections from the $B$-cosets in $A$ to the $A$-cosets in $B$. Otherwise, $\CC (A, B)$ consists of the empty bijection; 
\item[] $\CC (A, A)$ consists of the identity bijection on $A$ ;
\item[] morphism compositions is the usual composition of partial bijection. 
\end{itemize}\medskip

The category is modified in case \textbf{S} is just categorical by adding the requirement that, for each pair $A\geq B$, $\CC(A,B)$ contains the empty bijection.

We call $\CC $ the \textit{coset category}. By the nature of the coset bijections, $\CC $ and its dual category $\CC ^{op}$ coincide. On the other hand, due to the fact that coset bijections, when existing, are unique by Theorem \ref{coset_part}, $\vert \CC (A,B) \vert \leqslant 1$ and therefore this is a thin category. \medskip

Leech pointed out to the author that Theorem \ref{coset_part} together with Proposition 13 show us that the union of this family of bijections (where each bijection is seen as a set of ordered pairs) is equivalent to the order structure of $A\cup B$ as it is shown in the following:

\begin{theorem}
Let \textbf{S} be a categorical skew lattice and $A\geq B$ comparable $\DD $-classes. Then,
$$\cup \CC (A,B) = \cup \set{\varphi_{a,b}: a\in A,b\in B} =\geq _{A\times B},$$ where $\varphi _{a,b}$ is the coset bijection between $B\vee a\vee B$ and $A\wedge b\wedge A$.
\end{theorem}
 
\begin{proof} 
Just observe that all $\DD $-classes are anti-chains and therefore coset bijections describe all the order structure of the skew lattice. 
\end{proof} 
 
This last result reveals the important role of the coset structure in the study of skew lattices as it describes the order structure and provides us also the information on the structural effect of the congruence $\DD $. We conjecture an isomorphism between the coset category and the algebraic category for categorical skew lattices, though this is still an open problem. 

Moreover, the study of the coset category is yet to be done on a categorical perspective. Indeed, this is an unusual category that arises from the particular morphisms that it comprehends. Recent developments have been made in the study of the coset structure of a skew lattice with the characterization of several subvarieties of these algebras using coset identities in \cite{Jo09} and in \cite{Jo10} or the study of distributivity and cancellation in \cite{Ka08} and \cite{Le11}. As further the work on the coset structure of a skew lattice develops as more relevant is the challenge to study the coset category and strengthen the foundations of this new approach.

\section*{Acknowledgements}

We would like to thank to Karin Cvetko-Vah for her guidance in skew lattice Theory and her ideas and enthusiasm on many of the topics here in development. Also we also thank to Andrej Bauer for his orientation and questions in Category Theory topics as well as it's interpretation in skew lattice theory; to Jonathan Leech for the discussions on categorical skew lattices as well as his comments on the main topics here presented; to Margarita Ramalho for the inspiration for this approach; to Matja\v z Omladi\v c for his support; to Dijana Cerovski for her involvement; and to Mar\v cin Szamotulski for our short coffee break discussions on several "categorical" topics.

\end{document}